\documentclass[12pt]{article}
\usepackage{amsfonts}
\usepackage{amssymb}
\usepackage{amsmath}
\usepackage{amsthm}

\newtheorem{proposition}{proposition}[section]
\newtheorem{lemma}[proposition]{Lemma}

\newtheorem{theorem}[proposition]{Theorem}
\newtheorem{algorithm}[proposition]{Algorithm}

\date{1 April 2009}
\title{Three-coloring triangle-free planar graphs in linear time\footnote{An extended abstract of this paper appeared in ACM-SIAM Symposium on Discrete Algorithms (SODA'09). 
Published in {\em ACM Transactions on Algorithms} {\bf7} (2011), Article 41, 14 pages.}}
\author{Zden\v{e}k Dvo\v{r}\'ak\thanks{Department of Applied Mathematics,
Charles University, Prague, Czech Republic.}
\and
Ken-ichi Kawarabayashi\thanks{National Institute of Informatics,
2-1-2 Hitotsubashi, Chiyoda-ku, Tokyo 101-8430, Japan.}
\and
Robin Thomas\thanks{School of Mathematics,
Georgia Institute of Technology,
Atlanta, Georgia  30332-0160, USA.
Partially supported by NSF under Grant No.~DMS-0701077.}}

\begin{document}
\maketitle
\begin{abstract}

Gr\"otzsch's theorem states that every triangle-free planar graph is
$3$-colorable, and several relatively simple proofs of this fact were provided
by Thomassen and other authors.  It is easy to convert these proofs into
quadratic-time algorithms to find a $3$-coloring, but it is not clear how to
find such a coloring in linear time (Kowalik used a nontrivial data structure
to construct an $O(n \log n)$ algorithm).

We design a linear-time algorithm to find a $3$-coloring of a given
triangle-free planar graph.  The algorithm avoids using any complex data structures,
which makes it easy to implement.  As a by-product we give
a yet simpler proof of Gr\"otzsch's theorem.
\end{abstract}

\section{Introduction}
\label{intro}

The following is a classical theorem of Gr\"otzsch~\cite{Gro}.

\begin{theorem}
\label{grotzsch}
Every triangle-free planar graph is $3$-colorable.
\end{theorem}

This result has been the subject of extensive research.
Thomassen~\cite{ThoGro,ThoShortlist} found two short proofs and
extended the result in many ways. We return to the various extensions
later, but let us discuss algorithmic aspects of Theorem~\ref{grotzsch} first.
It is easy to convert either of Thomassen's proofs into a quadratic-time
algorithm to find a $3$-coloring, but it is not clear how to do so in
linear time.
A serious problem appears very early in the algorithm. Given a facial cycle $C$
of length four, one would like to identify a pair of diagonally opposite
vertices of $C$ and apply recursion to the smaller graph.
It is easy to see that at least one pair of diagonally opposite vertices
on $C$ can be identified without creating a triangle, but how can
we efficiently decide which pair?
If we could test in (amortized)
constant time whether given two vertices are joined
by a path of length at most three, then that would take care of this issue.
This can, in fact, be done, using a data structure of 
Kowalik and Kurowski~\cite{KowKur} {\em provided} the graph does not
change.
In our application, however, we need to repeatedly identify vertices,
and it is not clear how to maintain the data structure of
Kowalik and Kurowski in overall linear time.
Kowalik~\cite{Kow3col} developed a sophisticated enhancement of this
data structure that supports edge addition and deletion in amortized $O(\log n)$ time.
Furthermore, he found a variant of the proof of Gr\"otzsch's theorem
that can be turned into an $O(n\log n)$ algorithm to $3$-color a triangle-free planar
graph on $n$ vertices using this data structure.
We improve this to a linear-time algorithm, as follows.

\begin{theorem}
\label{main}
There is a linear-time algorithm to $3$-color an input triangle-free planar
graph.
\end{theorem}

\noindent
To describe the algorithm we 
exhibit a specific list of five reducible configurations, 
called ``multigrams",
and show that every triangle-free planar graph contains one of those
reducible configurations. 
Proving this is the only step that requires some effort; the rest of the
algorithm is entirely straightforward, and the algorithm is very
easy to implement.
Given a triangle-free planar graph $G$
we look for one of the reducible
configurations in $G$, and upon finding one we modify $G$ to a smaller
graph $G'$, and  apply the algorithm recursively to $G'$.
It is easy to see that every $3$-coloring of $G'$ can be
converted to a $3$-coloring of $G$ in constant time.
Furthermore, each reducible configuration has a vertex of degree at most
three, and, conversely, given a vertex of $G$
of degree at most three it can be
checked in constant time whether it belongs to a reducible configuration.
Thus at every step a reducible configuration can be found in 
amortized constant time
by maintaining a list of candidates for such vertices.
As a by-product of the proof of correctness of our algorithm we give
a short proof of Gr\"otzsch's theorem.

Let us briefly survey some of the related work.
Since in a proof of Theorem~\ref{grotzsch} it is easy to eliminate faces
of length four, the heart of the argument lies in proving the theorem
for graphs of girth at least five.
For such graphs there are several extensions of the theorem.
Thomassen proved in~\cite{ThoGro} that every graph of girth at least five
that admits an embedding in the projective plane or the torus is
$3$-colorable, and the analogous result for Klein bottle graphs was
obtained in~\cite{ThoWal}.
For a general surface $\Sigma$, Thomassen~\cite{ThoGirth5} proved the deep theorem
that there are only finitely many $4$-critical graphs of girth at 
least five that embed in $\Sigma$.
(A graph is $4$-critical if it is not $3$-colorable, but every proper
subgraph is.)

None of the results mentioned in the previous paragraph hold
without the additional restriction on girth.
Nevertheless, Gimbel and Thomassen~\cite{GimTho} found an elegant
characterization of $3$-colorability of triangle-free projective-planar
graphs.
That result does not seem to extend to other surfaces, but two of us
in a joint work with Kr\'al'~\cite{DvoKraTho}
 were able to find a sufficient condition for
$3$-colorability of triangle-free graphs drawn on a fixed surface $\Sigma$.
The condition is closely related to the sufficient condition for the existence
of disjoint connecting trees in~\cite{RobSeyGM7}.
Using that condition Dvo\v{r}\'ak, Kr\'al' and Thomas were able to
design a linear-time algorithm to test if a triangle-free graph on
a fixed surface is $3$-colorable~\cite{DvoKraTho}.

If we allow the planar graph $G$ to have triangles, then testing
$3$-colorability becomes NP-hard~\cite{GarJoh}.
There is an interesting conjecture
of Steinberg stating that every planar graph with no cycles of
length four or five is $3$-colorable,
but that is still open.
Every planar graph is $4$-colorable by the 
Four-Color Theorem~\cite{AppHak1, AppHakKoc, RobSanSeyTho4CT},
and a $4$-coloring can be found in quadratic time~\cite{RobSanSeyTho4CT}.
Any improvement to the running time of this algorithm would seem to
require new ideas.
A $5$-coloring of a planar graph can be found in linear 
time~\cite{NisChi}.

Our terminology is standard. 
All {\em graphs} in this paper are simple and 
{\em paths} and {\em cycles} have no repeated vertices.
By a {\em plane graph} we mean a graph that is drawn in the plane.
On several occasions we will be identifying vertices, but when we do,
we will remove the resulting parallel edges.
When this will be done by the algorithm we will make sure that
the only parallel edges that arise will form faces of length two.
The detection and removal of such parallel edges can be done in
constant time.

\section{Short proof of Gr\"otzsch's theorem}
\label{shortproof}

Let $G$ be a plane graph.  Somewhat nonstandardly, we call a cycle $F$ in $G$ {\em facial}
if it bounds a face in a connected component of $G$, regardless of whether $F$ is a face
or not (another component of $G$ might lie in the disk bounded by $F$).  
This technicality makes no difference in this section, because here we may assume
that all graphs are connected. However, it will be needed in the description of
the algorithm, because the graph may become disconnected during the course of
the algorithm, and we cannot afford to decompose it into connected components.

By a {\em tetragram} in $G$ we mean a sequence $(v_1,v_2,v_3,v_4)$
of vertices of $G$ such that they form a facial cycle in $G$ in the order
listed. We define a {\em hexagram} $(v_1,v_2,\ldots,v_6)$ similarly.
By a {\em pentagram} in $G$ we mean a sequence $(v_1,v_2,v_3,v_4,v_5)$
of vertices of $G$ such that they form a facial cycle in $G$ in the
order listed and $v_1,v_2,v_3,v_4$ all have degree exactly three.
We will show that every triangle-free planar graph of minimum degree
at least three has a
tetra-, penta- or hexagram with certain additional properties
that will allow an inductive argument.
But first we need the following lemma.

\begin{lemma}
\label{grams}
Let $G$ be a connected triangle-free plane graph
and let $f_0$ be the unbounded face of $G$.
Assume that the boundary of $f_0$ is
a cycle $C$ of length at most six,
and that every vertex of $G$ not on $C$ has degree
at least three.
If $G\ne C$, 
then $G$ has either a tetragram, or a pentagram $(v_1,v_2,v_3,v_4,v_5)$
such that $v_1,v_2,v_3,v_4\not\in V(C)$.
\end{lemma}

\begin{proof}
We define the charge of a vertex $v$ to be $3\deg(v)-12$,
the charge of the face $f_0$ to be $3|V(C)|+11$
and the charge of a face $f\ne f_0$ of length $\ell$ to be $3\ell-12$.
It follows from Euler's formula that the sum of the charges of all
vertices and faces is $-1$.

We now redistribute the charges according to the following rules.
Every vertex not on $C$ of degree three will receive one unit of charge 
from each incident face, 
each vertex on $C$ of degree three will receive three units from $f_0$,
and each vertex  of degree two on $C$ will receive 
five units from $f_0$ and one unit from the other incident face.
Thus the final charge of every vertex is non-negative. 

We now show that the final charge of $f_0$ is also non-negative.
Let $\ell$ denote the length of $C$.
Then $f_0$ has initial charge of $3\ell+11$.
By hypothesis at least one vertex of $C$ has degree at least three, and hence
$f_0$ sends
a total of at most $5(\ell-1)+3$ units of charge, leaving it at the
end with charge of at least $3\ell+11-5(\ell-1)-3\ge 1$.

Since no charge is lost or created, there is a face $f\ne f_0$ whose final
charge is negative.
Since $f$ sends at most one unit to each incident vertex, we see that 
$f$ has length at most five.
Furthermore, if $f$ has length exactly five, then it
 sends one unit to at least four incident vertices. None of those could
be a degree two vertex on $C$, for then $f$ would not be sending
anything to the  ends of the common subpath of the boundaries of $f$ and $f_0$.
Thus the vertices of $f$ form the desired tetragram or pentagram.
\end{proof}

Let $k=4,5,6$, and let $(v_1,v_2,\ldots,v_k)$ be a tetragram, pentagram
or hexagram in a triangle-free plane graph $G$. 
If $k=4$ or $k=6$, then we say that 
$(v_1,v_2,\ldots,v_k)$ is {\em safe} if every path in $G$ of length at most
three with ends $v_1$ and $v_3$ is a subgraph of 
the cycle $v_1v_2\cdots v_k$.
For $k=5$ we define safety as follows.
For $i=1,2,3,4$ let $x_i$ be the neighbor of $v_i$ distinct from
$v_{i-1}$ and $v_{i+1}$ (where $v_0=v_5$).
Then $x_i\not\in \{v_1,\ldots,v_5\}$, because $G$ is triangle-free.
Assume that

\begin{itemize}
\item the vertices $x_1,x_2,x_3,x_4$ are pairwise distinct
and pairwise non-adjacent, and
\item there is no path in $G\setminus\{v_1,v_2,v_3,v_4\}$
of length at most three from $x_2$ to $v_5$, and
\item every path in $G\setminus\{v_1,v_2,v_3,v_4\}$
of length at most three from $x_3$ to $x_4$ has length exactly two,
and its completion via the path $x_3v_3v_4x_4$ results in a facial cycle
of length five in $G$ (in particular, there is at most one such path).
\end{itemize}

In those circumstances we say that the pentagram $(v_1,v_2,\ldots,v_5)$
is {\em safe}.

\begin{lemma}
\label{safegrams}
Every triangle-free plane graph $G$ of minimum degree at least three
has a safe tetragram, a safe pentagram, or a safe hexagram.
\end{lemma}

\begin{proof}
Let $G$ be as stated.  If $(v_1,v_2,v_3,v_4)$ is a tetragram in $G$,
then one of the tetragrams $(v_1,v_2,v_3,v_4)$, $(v_2,v_3,v_4,v_1)$
is safe, as $G$ is planar and triangle-free.
Thus we may assume that $G$ has no $4$-faces, and hence every $4$-cycle in $G$
is separating.

Let us define an induced subgraph $G_1$ of $G$
and a facial cycle $C_1$ of $G_1$ in the following way:
If $G$ has a separating cycle of length at most five, then let us select
such a cycle $C_1$ so that the disk it bounds is as small as possible,
and let $G_1$ be the subgraph of $G$ consisting of all vertices and
edges drawn in the closed disk bounded by $C_1$.
If $G$ has no separating cycle of length at most five, then let $G_1:=G$
and let $C_1$ be a facial cycle of $G$ of length at most five.
Such a facial cycle exists, because
the minimum degree of $G$ is at least three.
In the latter case, we also redraw $G$ so that $C_1$ becomes the outer face;
thus $G_1$ is always drawn in the closed disk bounded by $C_1$.
Note that $G_1$ does not contain any separating cycle of length
at most five, and thus $G_1$ does not contain any $4$-cycle except
possibly $C_1$.

Next, we define a subgraph $G_2$ of $G_1$ and its facial cycle
$C_2$ as follows.  If $G_1$ contains a separating cycle of length six,
then choose such a cycle $C_2$ so that the disk it bounds contains as few vertices as possible,
and let $G_2$ be the subgraph of $G_1$ consisting of all vertices and
edges drawn in the closed disk bounded by $C_2$.  Otherwise, let $G_2:=G_1$
and $C_2:=C_1$.  Note that $G_2$ does not contain any separating
cycle of length at most six.  As $G$ has no $4$-faces, it follows that
any cycle of length at most six in $G_2$ bounds a face.

The cycle $C_2$ is induced in $G$,  
for if it had a chord, then the chord would belong to $G_1$ (because $G_1$ is an
induced subgraph of $G$), and hence $V(C_2)$ would include the vertex-sets 
of two distinct cycles of length at most (and hence exactly) four in $G_1$, a contradiction.

From Lemma~\ref{grams} applied to the graph $G_2$ and facial cycle $C_2$
we deduce that $G_2$ has a pentagram $(v_1,v_2,v_3,v_4,v_5)$
such that $v_1,v_2,v_3,v_4\not\in V(C_2)$.
We may assume that neither this pentagram nor the pentagram $(v_4,v_3,v_2,v_1,v_5)$
is safe in $G$, for otherwise the lemma holds.
Let $x_i$ be the neighbor of $v_i$ outside of the pentagram, for $1\le i\le 4$.
Note that all of these neighbors belong to $G_2$, and
as $G_2$ is triangle-free and contains no $4$-cycles other than $C_2$ 
 and no separating cycles of length at most $5$,
they are distinct and mutually non-adjacent.  It follows that $|\{x_1,x_2,x_3,x_4\}\cap V(C_2)|\le 3$,
and by symmetry we may assume that at least one of $x_3$ and $x_4$ does not lie on $C_2$.
Furthermore, as each cycle of length at most six in $G_2$ is facial, if $v_5\in V(C_2)$, then
$\{x_1,x_2,x_3,x_4\}\cap V(C_2)=\emptyset$.


Since the pentagram $(v_1,v_2,v_3,v_4,v_5)$ is not safe in $G$, there
exists a pair of vertices $x,y$ such that either $\{x,y\}=\{x_2,v_5\}$
or $\{x,y\}=\{x_3,x_4\}$, and there exists a path $P$ in 
$G\backslash\{v_1,v_2,v_3,v_4\}$ with ends $x$ and $y$ such that $P$
has length at most three, and if  $\{x,y\}=\{x_3,x_4\}$, then either
$P$ has length exactly three, or its completion via the path
$x_3v_3v_4x_4$ does not result in a facial cycle in $G$.
If $\{x,y\}=\{x_2,v_5\}$ then let $Q$ denote the path $x_2v_2v_1v_5$;
otherwise let $Q$ denote the path $x_3v_3v_4x_4$.
Suppose first that $P\cup Q$ bounds a face in $G$.
Then it follows that $\{x,y\}=\{x_3,x_4\}$, and hence $P$ has
length exactly three. Let the vertices of $P\cup Q$ be $x_3v_3v_4x_4ab$ 
in order.
Let us argue that $(x_4,v_4,v_3,x_3,a,b)$ is a safe hexagram.
If that were not the case, then there would exist a path
$x_4u_1v_3$ or $x_4u_1u_2v_3$ for some $u_1,u_2\neq v_4$.  Since $v_2$ and $v_3$
have degree three and the vertices $x_1$, $x_2$, $x_3$ and $x_4$ are distinct
and mutually non-adjacent, the former case is not possible, and in the latter
case $u_2=x_3$.  However, since at most one of $x_3$ and $x_4$ lies on $C_2$,
$x_4u_1x_3v_3v_4$ would be a separating $5$-cycle in $G_2$, and hence in $G_1$,
a contradiction.


Thus we may assume that $P\cup Q$ does not bound a face in $G$, and so 
$P\cup Q$ is a separating cycle in $G$.
It follows from the choice of $C_2$ that $P\cup Q$ is not a subgraph of $G_2$.
But not both $x,y$ belong to $C_2$ and $C_2$ is induced; 
thus a subpath $R$ of $P\cup Q$ of length four joins two vertices $w_1,w_4$
of $C_2$, and a vertex $w$ of $(P\cup Q)\backslash V(G_2)$ is adjacent to
both $w_1$ and $w_4$.
If $w\not\in V(G_1)$, then $w_1,w_4\in V(C_1)$, because they belong to $C_2$.
But $C_1$ has length at most five, and $w_1,w_4$ are not adjacent, because
$G$ is triangle-free.
Thus $w_1,w_4$ have a common neighbor in $C_1$, and this neighbor can
replace $w$. Thus we may assume that $w\in V(G_1)$.

If $w_1$ and $v_4$ have a common neighbor in $C_2$, then $R$ can be completed
using this neighbor to a cycle that contradicts the choice of $C_2$.
It follows that $w_1,w_4$ are at distance three on $C_2$, and so we
may assume that the vertices of $C_2$ are $w_1,w_2,\ldots,w_6$, in order.
From the symmetry we may assume that $w_1w_2w_3w_4w$ bounds a face,
by the minimality of $C_1$.
Thus the closed disk bounded by $P\cup Q$ does not include $w_5,w_6$,
and it includes no vertex of $V(G)-V(G_2)$, except $w$.
Thus $P\cup Q$ contradicts the choice of $C_2$.
\end{proof}

\begin{proof}[Proof of Theorem~\ref{grotzsch}.]
Let $G$ be a triangle-free plane graph.
We proceed by induction on $|V(G)|$.
We may assume that every vertex $v$ of $G$ has degree at least three,
for otherwise the theorem follows by induction applied to $G\setminus v$.
By Lemma~\ref{safegrams} there is a safe tetra-, penta-, or hexagram
$(v_1,v_2,\ldots,v_k)$. If $k=4$ or $k=6$, then we apply induction
to the graph obtained from $G$ by identifying $v_1$ and $v_3$.
It follows from the definition of safety that the new graph has
no triangles, and clearly every $3$-coloring of the new graph extends to
a $3$-coloring of $G$.

Thus we may assume that $(v_1,v_2,\ldots,v_5)$ is a safe pentagram in $G$.
Let $G'$ be obtained from $G\setminus\{v_1,v_2,v_3,v_4\}$ by identifying
$v_5$ with $x_2$, and $x_3$ with $x_4$.
It follows from the definition of safety that $G'$ is triangle-free,
and hence it is $3$-colorable by the induction hypothesis.
Any $3$-coloring of $G'$ can be extended to a $3$-coloring of $G$:
let $c_1$ be the color of $x_1$, $c_2$ the color of $x_2$ and $v_5$,
and $c_3$ the color of $x_3$ and $x_4$.  If $c_1=c_2$, then we color the
vertices $v_4$, $v_3$, $v_2$ and $v_1$ in this order.  Note that when
$v_i$ ($i=1,2,3,4$) is colored, it is adjacent to vertices of at most
two different colors, and hence we can choose the third color for it.
Similarly, if $c_2=c_3$, then we color the vertices in the following order:
$v_1$, $v_2$, $v_3$ and $v_4$.
Let us now consider the case that $c_1\neq c_2\neq c_3$.  We color $v_2$ with
$c_1$, $v_3$ with $c_2$, and choose a color different from $c_1$ and $c_2$
for $v_1$ and a color different from $c_2$ and $c_3$ for $v_4$.
Thus $G$ is $3$-colorable, as desired.
\end{proof}

Let us note that the essential ideas of the proof came
from Thomassen's work~\cite{ThoGro}.
For graphs of girth at least five Thomassen actually proves a stronger statement, 
namely that
every $3$-coloring of an induced facial cycle of length at most nine
extends to a $3$-coloring of the entire triangle-free plane graph,
unless some vertex of $G$ has three distinct neighbors on $C$
(and those neighbors received three different colors).
By restricting ourselves to Theorem~\ref{grotzsch} we were able to
somewhat streamline the argument.
Another variation of the same technique is presented in~\cite{Kow3col}.

\section{Graph representation}
\label{sec-repr}

For the purpose of our algorithm, graphs will be represented by means of
doubly linked adjacency lists.  More precisely, the neighbors of each
vertex $v$ will be listed in the clockwise cyclic order in which they
appear around $v$, and the two occurrences of the same edge will be
linked to each other. The facial walks of the graph can be read off from this
representation using the standard face tracing algorithm
(Mohar and Thomassen~\cite{mohthom}, page 93).  Thus all vertices and edges
incident with a facial cycle of length $k$ can be listed in time $O(k)$.
Here we make use of our non-standard definition of facial cycle.

Suppose that $D$ is a fixed constant (in our algorithm, $D=59$).
We can perform the following operations with graphs represented in the described way
in constant time:

\begin{itemize}
\item remove an edge when a corresponding entry of the adjacency list is given
\item add an edge with ends $u$, $v$ into a face $f$, assuming that the edges 
preceding and following $u,v$ in the facial boundary of $f$
are specified
\item remove an isolated vertex
\item determine the degree of a vertex $v$ if $\deg(v)\le D$, or prove that $\deg(v)>D$
\item check whether two vertices $u$ and $v$ such that $\min(\deg(u),\deg(v))\le D$ are adjacent
\item check whether the distance between two vertices $u$ and $v$ such that $\max(\deg(u),\deg(v))\le D$
is at most two
\item given an edge $e$ incident with  a face $f$, output all vertices 
whose distance from $e$ in the facial walk
of $f$ is at most two, and determine whether the length of the component of the boundary of $f$
that contains $e$ has length at most $6$
\item output the subgraph consisting of vertices reachable from a vertex $v_0$ through a path $v_0$, $v_1$, \ldots, $v_t$
of length $t\le D$, such that $\deg(v_i)\le D$ for $0\le i<t$ (but the degree of $v_t$ may be arbitrary).
\end{itemize}

All the transformations and queries executed in the algorithm can be expressed
in terms of these simple operations.

\section{The algorithm}
\label{sec-alg}

The idea of our algorithm is to find a safe 
tetragram, pentagram or hexagram $\gamma$ in $G$
and use it to reduce the size of the graph as in the proof of 
Theorem~\ref{grotzsch} above.
Finding $\gamma$ is easy, but the difficulty lies in testing safety.
To resolve this problem we prove a variant of Lemma~\ref{safegrams}
that will guarantee the existence of such $\gamma$ with an additional
property that will allow testing safety in constant time.
The additional property, called security,
 is merely that enough vertices in and around
$\gamma$ have bounded degree.
Unfortunately, the additional property we require necessitates the
introduction of two more configurations, a variation of  tetragram
called ``octagram'' and a variation of pentagram
called  ``decagram''. For the sake of consistency, we say that
a {\em monogram} in a graph $G$ is the
one-vertex sequence $(v)$ comprised of a vertex $v\in V(G)$ of degree 
at most two.

Now let $G$ be a plane graph, let $k\in\{1,4,5,6\}$ and let
$\gamma=(v_1,v_2,\ldots,v_k)$ be a mono-, tetra-, penta-, or hexagram in $G$.
Let $C$ be a subgraph of $G$.
(For the purpose of this section the reader may assume that $C$ is
the null graph, but in the next section we will need $C$ to be a
facial cycle of $G$.)
A vertex of $G$ is {\em big} if it has degree at least $60$,
and {\em small} otherwise.
A vertex $v\in V(G)$ is {\em $C$-admissible} if it is small and does
not belong to $C$; otherwise it is {\em $C$-forbidden}.
A pentagram $(v_1,v_2,\ldots,v_5)$ is called a {\em decagram}
if $v_5$ has degree exactly three (and hence $v_1,\ldots,v_5$ all
have degree three).  A tetragram is called an {\em octagram}
if all its vertices have degree exactly three.
A {\em multigram} is a monogram, tetragram, pentagram, hexagram, octagram or a decagram.
The vertex $v_1$ will be called the {\em pivot} of the multigram
$(v_1,v_2,\ldots,v_k)$.
In the following  $\gamma$ will be a multigram,
and we will define (or recall) what it means for
$\gamma$ to be safe and $C$-secure. 
We will also define a smaller graph $G'$, which will be called
the {\em $\gamma$-reduction of $G$}.

If $\gamma$ is a monogram, then we define it to be always {\em safe},
and we say that it is {\em $C$-secure} if $v_1\not\in V(C)$.
We define $G':=G\setminus v_1$.

Now let $\gamma$ be a tetragram.  Let us recall that $\gamma$ is
safe if the only paths in $G$ of length at most three 
with ends $v_1$ and $v_3$ are subgraphs of the facial cycle $v_1v_2v_3v_4$.
We say that $\gamma$ is {\em $C$-secure} if
\begin{itemize}
\item it is safe, and
\item $v_1$ is $C$-admissible and has degree exactly three, and
\item letting $x$ denote the neighbor of $v_1$ other than $v_2$ and $v_4$,
the vertex $x$ is $C$-admissible, and 
\item either 
\begin{itemize}
\item $v_3$ is $C$-admissible, or
\item every neighbor $w$ of $x$ is $C$-admissible or
belongs to a $4$-face incident with the edge $v_1x$ (either $v_1v_2wx$ or $v_1v_4wx$).
\end{itemize}
\end{itemize}
We define $G'$ to be the graph obtained from $G$ by identifying the
vertices $v_1$ and $v_3$ and deleting one edge from each of the two 
pairs of parallel edges that result.

If $\gamma$ is an octagram, then it is always {\em safe},
and it is {\em $C$-secure} if $v_1,v_2,v_3,v_4$ are all  are $C$-admissible.
We define $G':=G\setminus \{v_1,v_2,v_3,v_4\}$.

Now let $\gamma$ be a decagram, and for $i=1,2,3,4$ let $x_i$ be the
neighbor of $v_i$ other than $v_{i-1}$ or $v_{i+1}$, where $v_0$
means $v_5$.
We say that the decagram $\gamma$ is {\em safe} if $x_1,x_3$ are distinct,
non-adjacent and there is no path of length two between them.
We say that $\gamma$ is {\em $C$-secure} if it is safe
and the vertices $v_1,v_2,\ldots,v_5,x_1,x_3$ are all $C$-admissible.
We define $G'$ to be the graph obtained from $G\setminus\{v_1,v_2,\ldots,v_5\}$
by adding the edge $x_1x_3$.

Now let $\gamma$ be a pentagram, and for $i=1,2,3,4$ let $x_i$ be 
as in the previous paragraph.
Let us recall that the safety of $\gamma$ was defined prior to
Lemma~\ref{safegrams}.
We say that $\gamma$ is {\em $C$-secure} if it is safe, the vertices
$v_1,v_2,\ldots,v_5,x_1,x_2,x_3,x_4$ are all $C$-admissible,
either $v_5$ or $x_2$ has no $C$-forbidden neighbor,
and either $x_3$ or $x_4$ has no $C$-forbidden neighbor.
We define $G'$ as in the proof of Theorem~\ref{grotzsch}:
$G'$ is obtained from $G\setminus\{v_1,v_2,v_3,v_4\}$ by identifying
$x_2$ and $v_5$;  identifying $x_3$ and $x_4$; and deleting one of the
parallel edges should $x_3$ and $x_4$ have a common neighbor.

Finally, let $\gamma$ be a hexagram.
Let us recall that $\gamma$ is safe if every path of length at most three
in $G$ between $v_1$ and $v_3$ is the path $v_1v_2v_3$.
We say that $\gamma$ is {\em $C$-secure} if $v_1,v_3,v_6$ are $C$-admissible, 
$v_1$ has degree exactly three, and the neighbor of $v_1$ other than $v_2$ or
$v_6$ is $C$-admissible.
We define $G'$ to be the graph obtained from $G$ by identifying the
vertices $v_1$ and $v_3$ and deleting one of the parallel edges that result.

We say that a multigram $\gamma$ is {\em secure} if it is $K_0$-secure, where
$K_0$ denotes the null graph.
This completes the definition of safe and secure multigrams.

\begin{lemma}
\label{coloring}
Let $G$ be a triangle-free plane graph,  let $\gamma$ be a safe
multigram in $G$, and let $G'$ be the $\gamma$-reduction of $G$.
Then $G'$ is triangle-free, and
every $3$-coloring of $G'$ can be converted to a $3$-coloring of
$G$ in constant time.
Moreover, if $\gamma$ is secure, then
$G'$ can be regarded as having been obtained from $G$ by deleting
at most $126$ edges, adding at most $116$ edges, and deleting at least
one isolated vertex.
\end{lemma}

\begin{proof}
The graph $G'$ is triangle-free, because $\gamma$ is safe.
As in the proof of Theorem~\ref{grotzsch}, we argue that
every $3$-coloring of $G'$ can be extended to a $3$-coloring of $G$.
If $\gamma$ is secure, then
every time vertices $u$ and $v$ are identified in the construction of $G'$,
one of $u$, $v$ is small.
Thus the identification of $u$ and $v$ can be seen as a deletion
of at most $59$ edges and addition of at most $59$ edges.
The lemma follows by a more careful examination of the construction 
of $G'$.
\end{proof}

Let $G$ and $C$ be as above.  We say that two small vertices $u,v\in V(G)$
are {\em close} if either there is a path of length at most four between $u$ and $v$
consisting of small vertices, or a facial cycle of length at most six contains
both $u$ and $v$. A vertex $u$ is close to an edge $e$ if
both $u$ and $e$ belong to the facial walk of the same face and the distance
between $u$ and and one end of $e$ in this facial walk is at most two.
Thus for every vertex $v$ there are at most $1+4\cdot 59+59^2+59^3+59^4$
vertices that are close to $v$, and for every edge $e$, there are at most
$10$ vertices that are close to $e$.

\begin{lemma}
\label{findgram}
Given  a triangle-free plane graph $G$ and a vertex $v\in V(G)$,
it can be decided in  constant time whether $G$ has a secure multigram
with pivot $v$.
\end{lemma}

\begin{proof} 
This follows by inspecting the subgraph of $G$ induced by vertices and edges
that are close to $v$
and testing the security of all multigrams with pivot $v$ that lie in this subgraph.
Given such multigram, the only non-trivial part of testing security is testing
safety. Thus we may assume that the multigram satisfies all conditions in the
definition of security, except safety.
To test safety we need to check the existence of certain paths $P$ of bounded
length with prescribed ends.
We claim that whenever such a test is needed every vertex of $P$, 
except possibly one, is small.
The claim follows easily, except in
the case of a tetragram $vv_2v_3v_4$, where $v$ has degree three, the vertex $v_3$ is big,
and letting $x$ denote the neighbor of $v_1$ other than $v_2$ and $v_4$,
$x$ is small, but has a big neighbor $w$.  
In this case the straightforward check whether $w$ and $x_3$ are adjacent would take more
than constant time, but it actually follows that $w$ and $x_3$ are not adjacent:
the vertex $w$ belongs to a $4$-face incident with the edge $vx$, for otherwise
the tetragram is not secure; but then it follows  that $w$ and $x_3$ are not adjacent, for
otherwise $wv_3v_2$ would be a triangle. This proves our claim that
in the course of testing safety it suffices to examine paths with all but one
vertex small.

It follows from the claim that security can be tested in constant time, as desired.
\end{proof}

\begin{lemma}
\label{adddeledge}
Let $G$ and $G'$ be triangle-free plane graphs,  
such that for some pair of non-adjacent vertices $u,v\in V(G)$ the graph
$G'$ is obtained from $G$ by adding the edge $uv$.
Let $\gamma$ be a secure multigram in exactly one of the graphs $G,G'$.
Then the pivot of $\gamma$ is close to $u$ or $v$ in $G$,
or to the edge $uv$ in $G'$.
\end{lemma}

\begin{proof}
Let $v_1$ be the pivot of $\gamma$.  The claim is obvious if $v_1\in\{u,v\}$, and
thus assume this is not the case.  In particular, $\gamma$ is not a monogram or an octagram,
and $\gamma$ corresponds to a facial cycle $F$ in $G$ or $G'$.
If $F$ does not exist in $G$ or $F$ is not facial in $G$ or $G'$, then $v_1$ is close to the edge $uv$ in $G'$.  Let us now
consider the case that $F$ is a facial cycle both in $G$ and $G'$.
As $v_1\not\in \{u,v\}$, the degree of $v_1$ is three both in $G$ and $G'$.  Let $x_1$ be the
neighbor of $v_1$ distinct from its neighbors on $F$.
Note that $x_1$ is small in $G$.  

Suppose first that $\gamma$ is a tetragram or a hexagram.
Observe that the removal of the edge $uv$ from $G'$ must decrease the degree
of some of the vertices affecting the security of $\gamma$,
change the length of one of the faces incident with the edge $v_1x_1$
affecting the security of $\gamma$,
or destroy a path affecting its safety.  Therefore, if $\{u,v\}\cap (V(F)\cup\{x_1\})=\emptyset$
and $v_1$ is not close to the edge $uv$ in $G'$,
then $u$ or $v$ is a small neighbor of $x_1$ in $G$ that is big in $G'$.
We conclude that $v_1$ is close to $u$ or $v$ in $G$.

Let us now consider the case that $\gamma=(v_1,v_2,\ldots,v_5)$ is a decagram or a pentagram.
As $\gamma$ is secure in $G$ or $G'$, all the vertices of $\gamma$ are small in $G$.
If $\{u,v\}\cap V(F)\neq \emptyset$, then $v_1$ is close to $u$ or $v$ in $G$, and thus
assume that this is not the case.  It follows that the degree of $v_i$ is the same in $G$ and $G'$,
for $1\le i\le 5$; in particular, $\deg(v_i)=3$ for $1\le i\le 4$.
Let $x_i$ be the neigbor of $v_i$ not incident with $F$, for $1\le i\le 4$.
Similarly, we conclude that $x_1$ and $x_3$ are small in $G$, and if $\gamma$ is a pentagram,
then $x_2$ and $x_4$ are small in $G$.  If $\{u,v\}\cap \{x_1,x_3\}\neq\emptyset$, or
$\gamma$ is a pentagram and $\{u,v\}\cap \{x_2,x_4\}\neq\emptyset$, then $u$ or $v$
is close to $v_1$ in $G$.  If this is not the case, then the removal or addition of $uv$
cannot affect the security of $\gamma$ if $\gamma$ is a decagram.

We are left with the case when $\gamma$ is a pentagram, and $\{u,v\}\cap \{x_1,x_2,x_3,x_4\}=\emptyset$.
It follows that the neighborhoods of $x_2$, $x_3$, $x_4$ and $v_5$ are the same in $G$ and in $G'$.
As $\gamma$ is secure in $G$ or $G'$, all neighbors of $v_5$ or $x_2$, and all
neighbors of $x_3$ or $x_4$ are small in $G$.  As $\gamma$ is not secure both in $G$ and $G'$,
the removal of $uv$
\begin{itemize}
\item destroys a path of length at most three between $x_2$ and $v_5$ or
between $x_3$ and $x_4$, or
\item removes an edge incident with the common neighbor $y$ of $x_3$ and $x_4$,
thus making the $5$-cycle $x_3v_3v_4x_4y$ facial, or
\item decreases the degree of a neighbor of $x_2$, $x_3$, $x_4$ or $v_5$,
making it small in $G$.
\end{itemize}
In all the cases, $u$ or $v$ is a small neighbor of $x_2$, $x_3$, $x_4$ or $v_5$,
and hence it is close to $v_1$ in $G$.
\end{proof}

The next theorem will serve as the basis for the proof of correctness of our
algorithm.  We defer its proof until the next section.

\begin{theorem}
\label{multigramexists}
Every non-null triangle-free planar graph has a secure multigram.
\end{theorem}

We are now ready to prove Theorem~\ref{main}, assuming
Theorem~\ref{multigramexists}.

\begin{algorithm}
\label{alg:alg}
There is an algorithm with the following specifications:\\
{\em Input:} A triangle-free planar graph.\\
{\em Output:} A proper $3$-coloring of $G$.\\
{\em Running time:} $O(|V(G)|)$.
\end{algorithm}

\begin{proof}[Description.]
Using a linear-time planarity algorithm that actually outputs an embedding,
such as~\cite{ShiHsu} or~\cite{Wil}, we can assume that $G$ is a plane
graph.
The algorithm is recursive. 
Throughout the execution of the algorithm we will maintain a list
$L$ that will
include the pivots of all secure multigrams in $G$,
and possibly other vertices as well.
We initialize the list $L$ to consist of all vertices of $G$ of degree at
most three. 

At a general step of the algorithm we remove a vertex $v$ from $L$.
There is such a vertex by Theorem~\ref{multigramexists} and the requirement
that $L$ include the pivots of all secure multigrams.
We check if $G$ has a secure multigram with pivot $v$.
This can be performed in constant time by Lemma~\ref{findgram}.
If no such multigram exists, then we go to the next iteration.
Otherwise, we let $\gamma$ be one such multigram, and let
$G'$ be the $\gamma$-reduction of $G$.
By Lemma~\ref{coloring} the graph $G'$ is triangle-free
and can be constructed in constant time by adding and deleting bounded
number of edges, and removing a bounded number of isolated vertices.
For every edge $uv$ that was deleted or added during the construction
of $G'$ we add to $L$ all vertices that are close to $u$ or $v$, or to the
edge $uv$ in $G$ or $G'$.  
By Lemma~\ref{adddeledge} this will guarantee that $L$ will include
the pivots of all secure multigrams in $G'$.
We apply the algorithm recursively to $G'$,
and convert the resulting $3$-coloring of $G'$ to one of $G$
using Lemma~\ref{coloring}.
Since the number of vertices added to $L$ is proportional to the
number of vertices removed from $G$ we deduce that the number of
vertices added to $L$ (counting multiplicity) is at most linear
in the number of  vertices of $G$.
Thus the running time is $O(|V(G)|)$, as claimed.
\end{proof}

Algorithm~\ref{alg:alg} has the following extension.

\begin{algorithm}
\label{alg:extalg}
There is an algorithm with the following specifications:\\
{\em Input:} A triangle-free plane graph $G$, a facial cycle $C$ in
$G$ of length at most five, and a proper $3$-coloring $\phi$ of $C$.\\
{\em Output:} A proper $3$-coloring of $G$ whose restriction to $V(C)$
is equal to $\phi$.\\
{\em Running time:} $O(|V(G)|)$.
\end{algorithm}

\begin{proof}[Description.]
The description is exactly the same, except that we replace
``secure" by ``$C$-secure" and appeal to Lemma~\ref{lem:semigram}
rather than Theorem~\ref{multigramexists}.
\end{proof}

\section{Proof of correctness}
\label{sec:correct}

In this section we prove Theorem~\ref{multigramexists}, thereby
completing the proof of correctness of the algorithm from the previous
section.
The theorem will follow from the next lemma.
If $xy$ is an edge in a plane graph, and $f$ is a face of $G$ incident
with $y$ but not with the edge $xy$, then we say that {\em $f$ is
opposite to $xy$}.
Let us emphasize that this notion is not symmetric in $x,y$.

\begin{lemma}
\label{lem:semigram}
Let $G$ be a connected triangle-free plane graph and let $f_0$ be its
outer face.  Assume that $f_0$ is bounded by a cycle $C$ of length at most six,
$V(G)\neq V(C)$, and if $C$ has length six, then $|V(G)-V(C)|\ge 2$.
Then $G$ contains a $C$-secure multigram.
\end{lemma}
\begin{proof}
Suppose for a contradiction that the lemma is false, and let $G$ be
a counterexample with $|E(G)|$ minimum.
We first establish the following claim.

\begin{itemize}
\item[(1)] {\em If $K\neq C$ is a cycle in $G$ of length at most six,
   then $K$ bounds a face, or $K$ has length six and the open disk bounded by $K$
   contains at most one vertex.}
\end{itemize}

\noindent
To prove (1) let $K$ be as stated, and let $G'$ be the subgraph of $G$
consisting of all vertices and edges that belong to the closed disk
bounded by $K$.  If $K$ does not satisfy the conclusion of (1),
then $G'$ and $K$ satisfy assumptions of Lemma~\ref{lem:semigram}.
From the induction hypothesis applied to $G'$ and $K$ we deduce
that $G'$ has a $K$-secure multigram.
However, every $K$-secure multigram in $G'$ is a $C$-secure multigram in $G$.
\medskip

It follows from (1) that $C$ is an induced cycle and that every tetragram in $G$ is safe.

We assign charges to vertices and faces of $G$ as follows.
Initially, a vertex $v$ will receive a charge of $9\deg(v)-36$ if
$v\not\in V(C)$, and $8\deg(v)-19$ otherwise.
The outer face $f_0$ will receive a charge of zero, and every other face $f$
of length $\ell$ will receive a charge of $9\ell-36$.
By Euler's formula the sum of the charges is equal to
\begin{eqnarray*}
&&\sum_{v\not\in V(C)} 9(\deg(v)-4) + \sum_{v\in V(C)}(8\deg(v)-19)
  +\sum_{f\ne f_0}9(\,\hbox{size}(f)-4)\\
&=&\sum_{v\in V(G)} 9(\deg(v)-4)+ \sum_{f}9(\hbox{size}(f)-4) 
  -\sum_{v\in V(C)}\deg(v) + 8|V(C)| +36\\
&=&8|V(C)|-\sum_{v\in V(C)}\deg(v)-36\le-1,
\end{eqnarray*}
because all vertices of $C$ have degree at least two, and at least
one has degree at least three by hypothesis.
Furthermore,

\begin{itemize}
\item[(2)]
   {\em if at least $k$ vertices of $C$ have degree at least three,
   then the sum of the charges is at most~$-k$.}
\end{itemize}

We now redistribute the charges according to the following rules.
The new charge thus obtained will be referred to as the {\em final} charge.
We need a definition first.
Let $f\ne f_0$ be a face of $G$ incident with a vertex $v\in V(C)$.
If there exist two consecutive edges in the boundary of $f$ such that
both are incident with $v$ and neither belongs to $C$, then we say
that $f$ is a {\em $v$-interior face}.
The rules are:
\begin{description}
\item[(A)] every face other than $f_0$  sends three units of
charge to every incident vertex $v$ such that either $v\in V(C)$ and
$v$ has degree two in $G$, or $v\not\in V(C)$ and $v$ has degree exactly three,
\item[(B)] every big vertex not on $C$ sends three units to each incident face,
and four units to each $4$-face that shares an edge with $C$,
\item[(C)] every vertex  $v\in V(C)$ sends three units to every
$v$-interior face,
\item[(D)] if $x\in V(G)$ is $C$-forbidden, and $y$ is a $C$-admissible
neighbor of $x$ of degree three, then $x$ sends three units to the unique 
face opposite to $xy$, and one unit to the face opposite to $yz$
for every $C$-admissible neighbor $z$ of $y$ of degree three,
\item[(E)] every $C$-forbidden vertex sends five units to every
$C$-admissible neighbor of degree at least four,
\item[(F)] for every $C$-admissible vertex $y$ of degree at least four
that has a $C$-forbidden neighbor
we select a $C$-forbidden neighbor $x$ of $y$ and let $y$ send
one unit to each 
face opposite to $xy$, and one unit to the face opposite to $yz$
for every $C$-admissible neighbor $z$ of $y$ of degree three.
\end{description}

Since $G$ does not satisfy the conclusion of the theorem,
it follows that every vertex of $G$ has degree at least two, and
every vertex of degree exactly two belongs to $C$.
With these facts in mind we now show that every vertex has non-negative
charge.
To that end let $v\in V(G)$ have degree $d$, and assume first that $v$ is $C$-admissible.
If $d=3$, then it starts out with a charge of $-9$ and
receives three from each incident face by rule (A) for a final total of zero.
If $d\ge4$, then $v$ starts out with a charge of
$9d-36\ge0$. If $v$ has no $C$-forbidden neighbor, then it sends no charge
and the claim holds. Thus we may assume that $v$ has a $C$-forbidden
neighbor, and let $x$ be such neighbor selected by rule (F).
Then $v$ receives at least five units by rule (E), and sends
at most $2d-3$ by rule (F) for a total of at least
$9d-36+5-(2d-3)=7d-28\ge0$.
Thus every $C$-admissible vertex has non-negative final charge.
If $v$ is big, but does not belong to $C$, then it sends only by rules (B),
(D) or (E). It sends at most $3d$ using the first clause of rule (B),
at most $24$ using the second clause of rule (B) and at most $5d$ using rules (D) or (E)
for a total final charge of at least $9d-36-3d-24-5d\ge0$, because $d\ge 60$.
Thus we may assume that $v\in V(C)$. 
Then $v$  starts out with a charge of $8d-19$ and sends a net total of
$3(d-3)$ using rules (A) or (C)
(if $d=2$, then $v$ receives $3$ by rule (A); and otherwise it sends
$3(d-3)$ by rule (C)) and
it sends $5(d-2)$ using rule (D) or (E) for a total of
$8d-19-3(d-3)-5(d-2)=0$.
This proves our claim that the final charge of every vertex is non-negative.

It also follows that every face of length $\ell\ge6$ has
non-negative final charge, for every face sends at most three units
to each incident vertex and only to those vertices by rule (A); thus the final
charge is at most $9\ell-36-3\ell\ge0$.

We have thus shown that $G$ has a face $f$ of length at most
five with strictly negative final charge.
Clearly $f$ is not the outer face.

\begin{itemize}
\item[(3)] {\em No vertex incident with $f$ has degree two.}
\end{itemize}

\noindent
To prove (3) suppose for a contradiction that a vertex $v$ of degree
two is incident with $f$. Thus $v$ and the two edges incident with
$v$ and $f$ belong to $C$. Since $G\ne C$ and $f$ has length at most five
we deduce that at least two vertices incident with $f$ are incident with
$C$ and have degree at least three. 
Those two vertices do not receive any charge from $f$, and hence $f$
has length four, because it has negative charge.

We deduce that $f$ is bounded by a cycle $u_1u_2u_3u_4$, where
$u_1,u_2,u_3$ are consecutive vertices of $C$, and $u_2$ has degree two.
It follows that $u_4\not\in V(C)$, because $C$ is induced.
Since $f$ has negative charge it does not receive charge by rule (B), and hence $u_4$ is small and
$C$-admissible.
Let $C'$ be the cycle obtained from $C$ by replacing the vertex $u_2$
by $u_4$; note that $|V(C')|=|V(C)|\le 6$.  As $u_4$ has degree greater than
two, $C'$ does not bound a face, hence it follows from (1) that $|V(C')|=6$
and the open disk bounded by $C'$ contains at most one vertex.
Therefore, it contains exactly one, because $|V(G)|-V(C)|\ge 2$.
Let that vertex be $v_4$; then the remaining vertices of $C$ can be
numbered $v_1,v_2,v_3$ so that the cycle $C$ is $u_1u_2u_3v_1v_2v_3$
and $v_4$ is adjacent to $v_1$, $v_3$ and $u_4$.
Then $(u_4,u_1,u_2,u_3)$ is a $C$-secure tetragram,
contrary to the assumption that $G$ is a counterexample to the theorem.
This proves (3).
\medskip

Let $uv$ be an edge of $G$ such that $f$ is opposite to $uv$.
Let us say that $v$ is a 
{\em sink} if $v$ has degree three and both $u$ and $v$ are $C$-admissible.
Let us say that $v$ is a 
{\em source} if either $v\not\in V(C)$ and $v$ is big, or $v\in V(C)$ and 
$f$ is $v$-interior.
Since $v$ does not have degree two by (3) we deduce that
$v$ is a sink if and only if it has degree three and receives three units of charge from $f$
by rule (A) and $f$ does not receive three units by rule (D) 
from $u$.
Likewise, the vertex $v$ is a source if and only if it sends three
units to $f$ by the first clause of rule (B) or by rule (C).
Let $s$ be the number of sources, and $t$ the number of sinks.
Thus the charge of $f$ is at least $9+3s-3t$ if $f$ has length five
and at least $3s-3t$ if $f$ has length four.

Let us assume now that $f$ has length five, and let $v_1,v_2,\ldots,v_5$
be the incident vertices, listed in order. Since $f$ has negative charge,
at least four of the five incident vertices are sinks, and so
we may assume that $v_1,v_2,v_3,v_4$ are sinks. 
Thus $\gamma=(v_1,v_2,\ldots,v_5)$ is a pentagram.
For $i=1,2,3,4$ let $x_i$ be the neighbor of $v_i$ distinct
from $v_{i-1}$ and $v_{i+1}$ (where $v_0=v_5$).
From (1) and the fact that $G$ has no $C$-secure tetragram we deduce
that the vertices $x_1,x_2,x_3,x_4$ are distinct and pairwise
non-adjacent.
If $v_5$ is a $C$-admissible vertex of degree three, then
it follows from (1) that $\gamma$ is $C$-secure decagram---otherwise,
if there is a path of length two between $x_1$ and $x_3$, then consider
the $6$-cycle $K=x_1v_1v_2v_3x_3y$.  By (1) the open disk bounded by $K$ includes at most
one vertex of $G$.  It follows that $v_4$ and $v_5$ are not inside the disk;
thus either $y=x_2$ or $x_2$ is inside the disk.  In either case,
it follows that $x_2$ is adjacent to $x_1$ and $x_3$, a contradiction.
Thus $v_5$ is either not $C$-admissible, or has degree at least four.

Therefore, $v_5$ is not a sink, and hence the final charge of $f$ is
at least $-3$.
It follows that $v_5$ is not a source, which in turn implies that
$v_5$ is $C$-admissible (because $v_1$ and $v_4$ are $C$-admissible),
and hence has degree at least four.
We claim that $\gamma$ is a safe pentagram.
If there exists a path $P$ in $G\setminus\{v_1,v_2,v_3,v_4\}$
of length at most three with ends $x_2$ and $v_5$, then $P$ can be
completed to a cycle $K$ using the path $v_5v_1v_2x_2$.
By (1) we conclude that this cycle bounds an open disk that contains at most one
vertex, and it follows that $x_1$ is adjacent to $x_2$, which is a contradiction.
In order to complete the proof that $\gamma$ is safe it suffices
to consider a path in $G\setminus\{v_1,v_2,v_3,v_4\}$ 
of length at most three with ends $x_3$ and $x_4$.
This path can be completed via the path $x_4v_4v_3x_3$ to a cycle $K'$.
Since $v_3$ and $v_4$ have degree three, and $x_3$ is not adjacent to $x_4$,
we deduce from (1) that $K'$ is a facial cycle.
Since $x_3$ is not adjacent to $x_4$ 
we may assume for a contradiction that $K'$ has length six; let its vertices in order
be $x_3v_3v_4x_4ab$.
Then $(v_4,v_3,x_3,b,a,x_4)$ is a $C$-secure hexagram in $G$, a contradiction.
This proves our claim that $\gamma$ is a safe pentagram.
By symmetry the pentagram $(v_4,v_3,v_2,v_1,v_5)$ is also safe.
We have already established
that the vertices $v_1,v_2,\ldots,v_5,x_1,x_2,x_3,x_4$ are $C$-admissible.
If $x_i$ has a $C$-forbidden neighbor for some $i\in\{1,2,3,4\}$, then
$f$ receives one unit of charge either from that neighbor by rule (D)
if $x_i$ has degree three, or from $x_i$ by rule (F) otherwise.
Since the degree of $v_5$ is greater than three, if $v_5$ has a $C$-forbidden neighbor,
then it sends one unit of charge to $f$ by rule (F).
Thus at most two vertices among
$v_5,x_1,x_2,x_3,x_4$ have a $C$-forbidden neighbor, and hence
it follows that either $\gamma$, or
$(v_4,v_3,v_2,v_1,v_5)$ is a $C$-secure pentagram, a contradiction.

Thus we have shown that $f$ has length four.
Let $v_1,v_2,v_3,v_4$ be the incident vertices listed in order.
Let us recall that every tetragram is safe.
Since $f$ has negative charge at least $3s-3t$, 
we may assume that $v_1$ is a sink
and $v_3$ is not a source.
Since $v_3$ is not a source and $\gamma$ is not a $C$-secure
tetragram, $v_3\in V(C)$ and $f$ is not $v_3$-interior.
Then, (3) implies that exactly one of $v_2v_3$, $v_3v_4$ is an edge of $C$, and hence we may
assume the latter. In particular, $v_2\not\in V(C)$.
If $v_2$ is a sink, then the charge of $f$ is at least $-6$,
otherwise it is at least $-3$.

Let $v$ be the neighbor of $v_1$ other than $v_2$ and $v_4$.  Since
$v_1$ is a sink, $v$ is $C$-admissible.
If $v$ has no $C$-forbidden neighbor, then $\gamma$ is 
a $C$-secure tetragram, a contradiction.
Thus $v$ has a $C$-forbidden neighbor $u$.
Suppose first that $u\not\in V(C)$; hence $u$ is big and $f$ receives $4$ units of
charge from $u$ by rule (B).  As the charge of $f$ is negative,
we conclude that $v_2$ is a sink.  Let $v'$ be the neighbor of $v_2$ distinct
from $v_1$ and $v_3$. Since $\gamma$ is not a $C$-secure tetragram, $v'$
has a $C$-forbidden neighbor $u'$.  However, by rules (D) and (F),
$f$ receives one unit of charge from each of $u$ and $u'$,
making its final charge nonnegative.

We conclude that every $C$-forbidden neighbor of $v$ belongs to $C$.
Since rules (D) or (F) still apply, we obtain

\begin{itemize}
\item[(4)] {\em each $4$-face $f$ that shares an edge with $C$ has
final charge at least $-2t$, where $t\in\{1,2\}$ is the number of sinks of $f$.}
\end{itemize}

As $\gamma$ is not a $C$-secure tetragram, at least
one $C$-forbidden neighbor $u$ of $v$ is adjacent to neither $v_2$ nor $v_4$.
Let $C,C_1,C_2$ be the three cycles in the graph consisting of $C$
and the path $uvv_1v_4$, numbered so that $v_3$ belongs to $C_2$.
We claim that $C_2$ has length at least seven.  Note that $v_2$
lies in the open disk bounded by $C_2$; thus by (1) the cycle $C_2$ has length at least six.
Assume that $C_2$ has length exactly six. By (1), the open disk it bounds contains $v_2$ and no other
vertex of $G$. It follows that $v_2$ has degree three and is adjacent
to $u$, which contradicts the choice of $u$.

It follows that $C_2$ has length at least seven, and hence
$C_1$ has length at most five, and by the choice of $u$, it has length exactly five.
By (1), $C_1$ bounds a face.
Thus $u$ and $v_4$ have a common neighbor of degree two on $C$, say $z$.
Let $f(\gamma)$ denote the face bounded by $C_1$.  Let us call each tetragram
for which $f(\gamma)$ is defined {\em bad}.  Note that at this point, we have
proved that bad tetragrams are the only faces of $G$ with negative final
charge.  Let $b$ be the number of bad tetragrams.

The face $f(\gamma)$ starts out with a charge of $9$, sends three units
to each of $v_1,v,z$ by rule (A), and receives one 
either from $v_3$ by rule (D), or from $v_2$ by rule (F) for a total
of $+1$.  Also, if there exists a tetragram $\gamma'$ distinct
from $\gamma$ such that $f(\gamma)=f(\gamma')$, then the final charge
of $f(\gamma)$ is at least $+2$.  It follows that the total
charge of $G$ is at least $-b$.

Since $v_3$, $v_4$ and $u$ have degree at least three,
by (2) the total charge of $G$ is at most $-3$, and so $b\ge 3$.
However, since $b>1$, there must be another bad tetragram, giving
at least one more vertex of $C$ of degree at least three.
Therefore, the final charge of $G$ is at most $-4$ by (2), and hence $b\ge 4$.  
Let $u'$ be the unique neighbor of $u$ in $C\backslash z$. Since $b\ge4$
it follows by inspection that $v_3v_4$ and $uu'$ are the only edges of $C$
that belong to a bad tetragram.
We deduce that $G$ has a vertex $v'$ of degree three with neighbors $v,v_2,u'$.
It follows that $(v,v',v_2,v_1)$ is a $C$-secure octagram, as desired.
\end{proof}

\begin{proof}[Proof of Theorem~\ref{multigramexists}.]
Let $G$ be a triangle-free planar graph.
We may assume that $G$ is actually drawn in the plane.
If $G$ has a vertex of degree two or less, then it has a secure monogram,
and so we may assume that $G$ has minimum degree at least three.
It follows that $G$ has a facial cycle $C$ of length at most five.
Let $H$ be the component of $G$ containing $C$. We may assume that $C$
bounds the outer face of $H$. Since $H$ has minimum degree at least three
it follows that $V(H)-V(C)\ne\emptyset$.
By Lemma~\ref{lem:semigram} $H$ has a $C$-secure multigram;
but any $C$-secure multigram in $H$ is a secure multigram in $G$,
as desired.
\end{proof}

\centerline{\bf Acknowledgement}

We are indebted to a referee for carefully reading the manuscript and for
pointing out a couple of errors.


\begin{thebibliography}{99}

\def\JCTB{{\em J.~Combin.\ Theory Ser.\ B}}
\def\TAMS{{\em Trans.\ Amer.\ Math.\ Soc.}}

\bibitem{AppHak1} K. Appel and W. Haken, 
Every planar map is four colorable, Part I: discharging,
{\em Illinois J. of Math.} {\bf 21} (1977), 429--490.

\bibitem{AppHakKoc} K. Appel, W. Haken and J. Koch, 
Every planar map is four colorable, Part II: reducibility,
{\em Illinois J. of Math.} {\bf 21} (1977), 491--567.

\bibitem{DvoKraTho} Z.~Dvo\v{r}\'ak, D.~Kr\'al' and R.~Thomas,
Coloring triangle-free graphs on surfaces,
accepted to SODA'09.

\bibitem{GarJoh} M.~R.~Garey and D.~S.~Johnson, 
Computers and intractability. A guide 
to the theory of NP-completeness, W. H. Freeman, San Francisco, 1979.

\bibitem{GimTho} J.~Gimbel and C.~Thomassen,
Coloring graphs with fixed genus and girth,
\TAMS\ {\bf349} (1997), 4555--4564.

\bibitem{Gro} H.~Gr\"otzsch, 
Ein Dreifarbensatz f\"ur dreikreisfreie Netze auf der Kugel,
{\em Wiss. Z. Martin-Luther-Univ. Halle-Wittenberg Math.-Natur. Reihe}
{\bf 8} (1959), 109--120.

\bibitem{Kow3col} \L.~Kowalik,  
Fast 3-coloring triangle-free planar graphs,  
Algorithms---ESA 2004,  436--447, {\em Lecture Notes in Comput.\ Sci.}
{\bf 3221}, Springer, Berlin, 2004.

\bibitem{KowKur} \L. Kowalik and M. Kurowski, 
Oracles for bounded length shortest paths in planar graphs,
{\em ACM Trans.\ Algorithms}  {\bf2}  (2006),  335--363.

\bibitem{mohthom} B.~Mohar and C.~Thomassen,
Graphs on Surfaces,
The Johns Hopkins University Press, Baltimore and London, 2001.

\bibitem{NisChi} T.~Nishizeki and N.~Chiba,
Planar graphs: theory and algorithms,
{\em Ann.\ Discr.\ Math.} {\bf 32}, North-Holland, Amsterdam, 1988.

\bibitem{RobSanSeyTho4CT}
N.~Robertson, D.~P.~Sanders, P.~D.~Seymour and R.~Thomas,
The four-colour theorem, \JCTB\ {\bf 70} (1997), 2-44.

\bibitem{RobSeyGM7} N. Robertson and P. D. Seymour, 
Graph Minors VII. Disjoint paths on a surface, 
\JCTB\ {\bf45} (1988), 212--254.

\bibitem{ShiHsu}W.-K. Shih and W.-L. Hsu, 
A new planarity test,
{\em Theoret.\ Comp.\ Sci.\ \bf 223} (1999), 179--191.

\bibitem{ThoWal} R.~Thomas and B.~Walls, 
Three-coloring Klein bottle graphs of girth five,
\JCTB\ {\bf  92}  (2004), 115--135.

\bibitem{ThoGro} C.~Thomassen,
Gr\"otzsch's $3$-color theorem and its counterparts
for the torus and the projective plane,
\JCTB\ {\bf62} (1994), 268--279.

\bibitem{ThoShortlist} C. Thomassen, 
A short list color proof of Gr\"otzsch's theorem,
\JCTB\ {\bf88} (2003), 189--192.

\bibitem{ThoGirth5} C. Thomassen,  
The chromatic number of a graph of girth 5 on a fixed surface,
\JCTB\ {\bf87} (2003), 38-71.

\bibitem{Wil}S.~G.~Williamson, 
Depth-first search and Kuratowski subgraphs,
{\em J. Assoc. Comput. Mach.} {\bf 31} (1984), 681--693.

\end{thebibliography}
\end{document}